\documentclass[reqno]{amsart}
\usepackage{amsmath, amssymb, amsthm, epsfig,latexsym}

\usepackage{color}

 \newtheorem{theorem}{Theorem}
 \newtheorem{lemma}{Lemma}

 \theoremstyle{definition}

  \theoremstyle{remark}

  \newcommand{\mc}{\mathcal}
 
 \newcommand{\C}{\mathbb{C}}
 \newcommand{\R}{\mathbb{R}}
 \newcommand{\N}{\mathbb{N}}
 
 \newcommand{\Z}{\mathbb{Z}}

\numberwithin{equation}{section}

 \def\today{\ifcase\month\or
  January\or February\or March\or April\or May\or June\or
  July\or August\or September\or October\or November\or December\fi
  \space\number\day, \number\year}

\title{Extremal Signatures}
\author{Friedrich Littmann \and Mark Spanier}

\date{\today}

\begin{document}

\begin{abstract} Let $E= A-iB$ be a Hermite-Biehler entire function of exponential type $\tau/2$ where $A$ and $B$ are real entire, and consider $d\mu(x) = dx/|E(x)|^2$. We show that the sign of the product $A B$ is an extremal signature for the space of functions of exponential type $\tau$ with respect to the norm of $L^1(\mu)$. This allows us to find best approximations  by entire functions of exponential type $\tau$  in $L^1(\mu)$-norm to certain special functions (e.g., the Gaussian and the Poisson kernel).  
\end{abstract}



\maketitle

\section{Introduction}

 An entire function $F$ is said to be of \textsl{exponential type} $\tau\ge 0$ if
\[
|F(z)|\lesssim e^{(\tau+\varepsilon)|z|}
\]
for every $\varepsilon>0$ and $z\in\C$. (The implied constant may depend on $\varepsilon$, but not $z$.)  We denote by $\mc{B}(\tau)$ the class of entire functions of exponential type $\tau$.  For a Borel measure $\mu$ on $\R$ and $1\le p\le \infty$ we define
\[
\mc{B}_p(\mu,\tau) =\mc{B}(\tau)\cap L^p(\mu).
\]

We write $\mc{B}_p(\tau)$ if $\mu$ is Lebesgue measure.  

\medskip

Let $\tau>0$ and $\mu$ a Borel measure on $\R$. We are interested in {\it high pass functions for } $\mc{B}_1(\mu,\tau)$, that is, bounded, $\mu$-measurable $\psi:\R\to \C$ such that
\begin{align}\label{intro-es}
\int_\R F(x) \psi(x) d\mu(x) =0
\end{align}
for all $F\in \mc{B}_1(\mu,\tau)$. We denote by $\mc{A}(\mu,\tau)$ the class of high pass functions for $\mc{B}_1(\mu,\tau)$. (We use the letter $\mc{A}$ since these functions correspond to the class of integration functionals on $L^1(\mu)$ that annihilate $\mc{B}_1(\mu,\tau)$.) Of particular interest to us is the subclass
\[
\mc{S}(\mu,\tau) = \{\psi\in\mc{A}(\mu,\tau) : |\psi|=1 \text{ $\mu$-a.e.}\}.
\]

 The elements of $\mc{S}(\mu,\tau)$ will be called \textsl{extremal signatures for }$\mc{B}_1(\mu,\tau)$ or simply extremal signatures, if measure and type are clear from the context. The following connection between extremal signatures and best approximation is well known and explains why the study of high pass signatures is relevant. For $w=re^{i\theta}$ with $r>0$ and $0\le \theta<2\pi$ we define $\text{sgn}(w) = e^{i\theta}$, and we set $\text{sgn}(0) =0$. 
 
\medskip

\noindent{\bf Theorem A} \cite[Theorem 1.7, $6^\circ$]{Sing}. {\it
Let $f\in L^1(\mu)$, and define for $F\in \mc{B}_1(\mu,\tau)$ a function $\psi= \psi_{f,F}\in L^\infty(\R)$ by 
\[
\psi(x) = \overline{{\rm sgn}}(f(x) - F(x)).
\]

If $\psi\in \mc{S}(\mu,\tau)$, then 
\[
\|f - F\|_1\le \|f-G\|_1
\]
for all $G\in \mc{B}_1(\mu,\tau)$. (We say in this case that $F$ is a best approximation to $f$ in $L^1(\mu)$-norm.)
}

\medskip

The measures $\mu$ that we consider are defined below in \eqref{intro-mu} and the extremal signatures are given in \eqref{intro-psi}.

By way of motivating our results we describe briefly the classical case (essentially due to M.G.\ Krein \cite{Krein2}) when $\mu$ is the Lebesgue measure. In this case, extremal signatures have an equivalent formulation in terms of the distributional Fourier transform $\widehat{\psi}$ of $\psi$. By the Paley-Wiener theorem 
\[
\mc{B}_1(\tau) = \{F\in L^1(\R): \widehat{F}(t)=0\text{ for }|t|\ge \tau/(2\pi)\}.
\]

Equation \eqref{intro-es} then immediately implies that $F$ is a best approximation to $f$ in $L^1(\R)$ if and only if $\widehat{\psi}$ is supported in $\R\backslash (-\frac{\tau}{2\pi}, \frac{\tau}{2\pi})$. One such signature is  $\psi_0$ defined by
\begin{align}\label{intro-sine-sig}
\psi_0(x) = \text{sgn}\sin \tau x.
\end{align}

It can be seen from the Fourier series expansion of $\psi_0$ that $\psi_0$ is an extremal signature for $\mc{B}_1(\tau)$. This leads to the classical constructive approach of finding best approximations for $f$: a candidate for the best approximation to $f$ can be obtained by considering an element $F\in \mc{B}_1(\tau)$ that satisfies $f(x_0+\pi n/\tau) = F(x_0+\pi n/\tau)$ for some $x_0\in\R$ and all $n\in\Z$. If $f-F$ has sign changes at the values $x = x_0+\pi n/\tau$ and nowhere else on the real line, then by the previous discussion $F$ will be a best approximation to $f$. This construction succeeds for many elementary functions, e.g., the Gaussian  $G_\lambda(x) = e^{-\lambda x^2}$, the kernel $g_\lambda(x) = e^{-\lambda |x|}$, or the Poisson kernel $P_\lambda(x) = \frac{\lambda}{\pi} (x^2 + \lambda^2)^{-1}$. These cases are contained in the classical theorem of B.\ Sz.-Nagy \cite{N}, see also the articles by O.L.\ Vinogradov \cite{Vi} and  M.\ Ganzburg \cite{Ga}. 

 We  are interested in extremal signatures that play an analogous role  for $\mc{B}_1(\mu,\tau)$ as the signature $\psi_0$ in \eqref{intro-sine-sig} plays for $\mc{B}_1(\tau)$. (The starting point of our investigation was interest in the situation for the measure  $dx/(1+x^2)$.) Thus, in addition to an extremal signature we need an interpolation method that creates entire approximations so that the difference of approximation and function has sign changes only at the sign changes of the signature. 
 
Recent investigations \cite{CL}, \cite{CL2} into the Beurling-Selberg extremal problem (essentially the problem of best onesided $L^1(\mu)$-approximation) utilized an interpolation method with the following properties. For given  $f$ (e.g., the Gaussian or the Poisson kernel) and a Laguerre-P\'olya entire function $L$, the interpolation method creates under mild conditions an entire function $G= G_{f,L}$ such that
 \[
 f(x) - G(x) = L(x) H(x)
 \]
 on the real line, where $H$ is of one sign on $\R$ and $H(x) \lesssim (1+x^2)^{-1}$. The signature of $f-G$ is therefore completely determined by $L$, and the growth estimate for $H$ together with the requirement that $f-G\in L^1(\mu)$ gives an integrability condition that $L$ needs to satisfy.  
 
 In order to obtain the analogues for \eqref{intro-sine-sig}, we are therefore led to seek for a Laguerre-P\'olya entire function $L$ of exponential type $\tau$ with the following properties. 
  \begin{enumerate}
 \item The sign of $L$ is an element of $\mc{S}(\mu,\tau)$,
 \item the function $L$ satisfies
 \[
 \int_\R \frac{L(x)}{1+x^2} d\mu(x)<\infty.
 \]
 \end{enumerate}

We succeed at this task for measures $\mu= \mu_M$ given by
\begin{align}\label{intro-mu}
\mu_M(A) = \int_A \frac{dx}{M(x)},
\end{align}
where $M$ is an entire function of exponential type $\tau_0\ge 0$ that is positive on the real line and can be written as the quotient of two bounded analytic functions in the upper half plane. (Measures of this form occur in a theory developed by L.\ de Branges \cite{B}; Section \ref{dBspace} includes  summary of the required notions of de Branges' theory of Hilbert spaces of entire functions.)

The number $\tau_0$ associated with \eqref{intro-mu} is fixed, and the following results hold for all $\tau\ge \tau_0$.  We construct in  Section \ref{es-section} an associated entire function $E_{\tau/2}$ of exponential type $\tau/2$ with the property that
\begin{align}\label{meas-intro}
\int_\R \frac{F(x)}{M(x)} dx = \int_\R \frac{F(x)}{|E_{\tau/2}(x)|^2 }dx
\end{align}
for all $F \in L^2(\mu_M)$. Writing $A_{\tau/2}$ and $B_{\tau/2}$ for the unique real entire functions such that $E_{\tau/2} = A_{\tau/2} - iB_{\tau/2}$, we prove in Theorem \ref{es-theorem} that $\psi$ defined by
\begin{align}\label{intro-psi}
\psi(x) = {\rm sgn} (A_{\tau/2}(x) B_{\tau/2}(x))
\end{align}
is  an extremal signature for $\mc{B}_1(\mu_M,\tau)$. Furthermore, we indicate in Section \ref{LP-section} that the interpolation method mentioned above allows for any $f\in \{G_\lambda, g_\lambda, P_\lambda\}$ the construction of  $F_f\in\mc{B}_1(\mu_M, \tau)$ such that
\[
{\rm sgn}(F_f-f) = {\rm sgn}(A_{\tau/2} B_{\tau/2}).
\]

 An appeal to Theorem A implies that $F_f$ must necessarily be a best approximation from $\mc{B}_1(\tau,\mu)$ to $f$ in $L^1(\mu)$-norm. In this sense \eqref{intro-psi} is the correct generalization of \eqref{intro-sine-sig}. For the Poisson kernel and the conjugate Poisson kernel we obtain in addition in Theorem \ref{poisson-approx} a fairly explicit representation of the approximation error in terms of $E_{\tau/2}$. 

We note that  if $M\equiv 1$, then the construction in Section \ref{es-section} gives $E_{\tau/2}(z) = e^{-i(\tau/2) z}$. Inserting this in Theorem \ref{es-theorem} leads to $\psi(x) = {\rm sgn}(\cos(\tau x/2)  \sin(\tau x/2) ) = \psi_0(x)$, hence the classical results for Lebesgue measure are recovered. 

 It is crucial for our method that $d\mu_M = dx/|E_{\tau/2}|^2$. Requiring \eqref{meas-intro} only for all $F$ in the de Branges space $\mc{H}(E_{\tau/2})$ would not suffice. This is the reason why we can handle only measures of the form \eqref{intro-mu}. It is worth emphasizing that in this case the structure of the $L^2(\mu)$-space for exponential type $\tau/2$ determines the form of the best approximation of exponential type $\tau$ in $L^1(\mu)$.

We mention finally that the notion of an extremal signature occurs also in $L^\infty$-approximation problems. The theory related to this type of signature differs significantly from $L^1$ signatures, and we do not touch on this here. For readers interested in this direction we refer to A.\ Eremenko and P.\ Yuditskii \cite{EY} and Y.\ Xu \cite{X}.


\section{Entire functions of bounded type}\label{dBspace}

In this section we review known facts about the Nevanlinna class, Hardy spaces, and de Branges spaces. (Readers familiar with this material may safely skip this section.)  For convenience we include sketches of the shorter proofs and references to the longer proofs. Sources for this section are the books by M.\ Rosenblum and J.\ Rovnyak \cite{MR}, J.\ Levin \cite{Le}, and L.\ de Branges \cite{B}.

For any entire function $F$ we use the notation $F^*(z) = \overline{F(\overline{z})}$.  For nonnegative $u$ we set $\log^+(u) = \max(\log u,0)$. An analytic function $F$ defined in the upper half plane $\C^+$ is said to have bounded type, if $F$ is the quotient of two functions that are analytic and bounded in $\C^+$, and the denominator function is non-zero in $\C^+$. (This is sometimes called the Nevanlinna class of $\C^+$.) By \cite[Theorem 6.13]{MR} the function $F$ has the factorization 
\begin{align}\label{Ne-fac}
F(z) = e^{-i\nu z} B(z)\frac{S_+(z)}{S_-(z)} G(z)
\end{align}
in $\C^+$, where $\nu$ is real, $B$ is a Blaschke product, $S_+$, and $S_-$ are exponentials of Herglotz transforms of singular positive measures, and $G$ is an outer function. The number $\nu = \nu(F)$ is called the mean type of $F$ in $\C^+$; one of the formulas for its calculation is 
\begin{align}\label{type-calc}
\nu(F) = \limsup_{y\to \infty} y^{-1} \log|F(iy)|. 
\end{align}

By a classical theorem of M.G.\ Krein there is a close connection between the mean types of $F$ in the upper and lower half plane (the latter being also $\nu(F^*)$) and exponential type.

\begin{lemma}\label{krein-lemma} Let $F:\C\to \C$ be an entire function. The following are equivalent.
\begin{enumerate}
\item $F$ and $F^*$ have bounded type in the upper half plane. 
\item $F$ has exponential type and
\begin{align}\label{krein-ineq}
\int_\R \frac{\log^+|F(x)|}{1+x^2} dx <\infty.
\end{align}

If either (and therefore both) of these conditions hold, then the exponential type $\tau$ of $F$ satisfies $\tau\le \max(\nu(F), \nu(F^*))$.
\end{enumerate}
\end{lemma}

\begin{proof} This was shown first by Krein in \cite{Krein}; see also \cite[Theorem 6.17]{MR}.
\end{proof}

We denote by $\mc{C}$ the class of entire functions satisfying either (and hence both) of the conditions in Lemma \ref{krein-lemma}. We require the following decomposition result essentially due to N.\ Akhiezer.

\begin{lemma}\label{akhiezer} Let $F\in\mc{C}$ satisfy $F(x)\ge 0$ for all $x\in\R$. Denote by $\tau$ the exponential type of $F$. Then there exists an entire function $U\in \mc{C}$ of exponential type $\tau/2$ having zeros only in the closed lower half plane such that $F = UU^*$. 
\end{lemma}

\begin{proof} Denote by $\mc{A}$ the class of entire functions $G$ such that
\[
\sum_{k=1}^\infty \left| \Im \frac{1}{z_k} \right| <\infty,
\]
where $z_k$ denotes the nonzero zeros of $G$ (listed with multiplicity). By a theorem of Cartwright an entire function $F\in\mc{C}$ of exponential type is in $\mc{A}$, and by a theorem of Akhiezer  a nonnegative function $F\in\mc{A}$ of exponential type $\tau$ can be written as $F = UU^*$ where $U$ is an entire function of exponential type $\tau/2$ with zeros in $\Im z\le 0$. Since $F$ satisfies \eqref{krein-ineq}, $U$ satisfies \eqref{krein-ineq} as well.  (See, e.g., Theorem 7 on page 243 and Theorem 1 on page 437 of \cite{Le}.)
\end{proof}

 An entire function $E$ satisfying the inequality
\begin{align}\label{HB}
|E(z)|>|E^*(z)|
\end{align}
for all $z$ with $\Im z >0$ will be called a \textsl{Hermite-Biehler} (HB) function. We denote by $\mc{H}(E)$  the vector space of entire functions $F$ such that
\begin{align}\label{H2E}
\int_{-\infty}^\infty |F(x)/E(x)|^2 dx <\infty
\end{align}
and the functions $F/E$ and $F^*/E$ have bounded type and nonpositive mean type in $\C^+$. It is a Hilbert space with scalar product
\[
\langle F, G \rangle_E = \int_{-\infty}^\infty F(x) G^*(x) |E(x)|^{-2} dx.
\]

 We define $A = (1/2)(E+E^*)$ and $B = (i/2)(E - E^*)$. A fundamental result of de Branges from the 1960's is the recognition that this space is a reproducing kernel Hilbert space; we briefly sketch the argument. It follows from \eqref{HB} that $E$ has no zeros in the open upper half plane, and it follows from the definition of $\mc{H}(E)$ that $F/E$ and $F^*/E$ have no zeros in an open set containing the closed upper half plane. The condition that $F/E$ and $F^*/E$ have non-positive mean type implies that the Cauchy integral formula for the upper half plane holds for $F/E$ and $F^*/E$ (e.g., \cite[Theorem 12]{B}, note also that Cauchy formulas for $F/E^*$ and $F^*/E^*$ in the lower half plane hold), and it follows with an elementary residue calculation that $\mc{H}(E)$ has the reproducing kernel $K$ given by
\begin{align}\label{K-from-AB}
K(w,z) = \frac{B(z) A(\bar{w}) - A(z) B(\bar{w})}{\pi(z-\bar{w})}
\end{align}
for $z\neq \bar{w}$.

We denote by $\mc{H}^p(E)$ the vector space that is obtained by replacing the $L^2$-norm in \eqref{H2E} by the $L^p$-norm. (We require $p=1$ and $p=2$.)

The above definitions can be recast in the language of Hardy spaces. We recall that the Hardy space $H^p(\C^+)$ is the space of functions $F$ that are analytic in $\C^+$ such that the horizontal norms $\|F(.+iy)\|_p$ are uniformly bounded for $y>0$. In particular, $H^\infty(\C^+)$ is the space of bounded analytic functions in the upper half plane.

\begin{lemma}\label{HpE-Lp-equivalent} Let $E$ be a HB function that has bounded type in $\C^+$, and let $F$ be an entire function. For $1\le p<\infty$ the following conditions are equivalent.
\begin{enumerate}
\item $F$ belongs to $\mc{H}^p(E)$,
\item $F$ has exponential type, $\max(\nu(F), \nu(F^*))\le \nu(E)$, $F/E\in L^p(\R)$, and
\[
\int_\R \frac{\log^+|F(x)|}{1+x^2} dx <\infty,
\]
\item $F$ has exponential type $\tau\le \nu(E)$ and $F/E\in L^p(\R)$.
\end{enumerate}
\end{lemma}

\begin{proof} We follow J.\ Holt and J.D.\ Vaaler \cite[Lemma 12]{HV}.  The definition of $\nu(F)$ implies that $\nu(FG) = \nu(F) + \nu(G)$, and \eqref{type-calc} shows $\nu(F+G)\le \max(\nu(F), \nu(G))$.

The assumption $F\in \mc{H}^p(E)$ implies that $\nu(F/E)\le 0$. We obtain $\nu(F)= \nu(E (F/E))\le \nu(E)$, and similarly $\nu(F^*)\le \nu(E)$. Jensen's inequality gives
\[
\int_\R (1+x^2)^{-1} \log^+|F(x)/E(x)| dx\le \frac{\pi}{p} \log\left(1+\frac{1}{\pi} \|F/E\|_p^p\right)<\infty,
\]
and since $E$ satisfies \eqref{krein-ineq} by assumption, $F$ satisfies \eqref{krein-ineq} as well. This proves that (1) implies (2). Lemma \ref{krein-lemma} shows that (2) implies (3).  Since $E$ is HB, it has no zeros in the upper half plane, and hence $1/E$ is of bounded type in $\C^+$. It follows from \eqref{Ne-fac} that $\nu(1/E) = -\nu(E)$, hence (3) implies that $\nu(F/E)\le 0$, which is needed to show (1).
 \end{proof}

\begin{lemma}\label{HpE-HpC-equivalent} Let $1\le p\le \infty$, and let $E$ be a Hermite-Biehler entire function without real zeros. The following conditions are equivalent.
\begin{enumerate}
\item $F\in \mc{H}^p(E)$,
\item $F/E$ and $F^*/E$ are in $H^p(\C^+)$. 
\end{enumerate}
\end{lemma}

\begin{proof} The lemma is a fairly direct consequence of \eqref{Ne-fac} (e.g., \cite[Proposition 8]{G}). We note that $F/E$ and $F^*/E$ are analytic on a set containing the closure of $\C^+$. 

Assume first that $F/E \in H^p(\C^+)$. If $p=\infty$, then the quotients are bounded in $\C^+$ and hence have bounded type and non-positive mean type. For $1\le p<\infty$ we note \cite{MR} that $F/E = AG$ where $A$ is an inner function and $G$ is an outer function. Comparing with \eqref{Ne-fac} it follows in particular that the exponential $e^{-i\tau z}$ is part of the inner function, which is to say that $F/E$ has non-positive mean type. The same argument works for $F^*/E$. 

For the other direction, assume that $F/E$ has bounded type and mean type $\tau\le 0$, and $F/E\in L^p(\R)$. An estimate involving the Poisson integral of $\log |F/E|$ and \eqref{Ne-fac} leads with an application of Jensen's inequality to the inequality
\[
\left| \frac{F(x+iy)}{E(x+iy)}\right| \le \frac{y}{\pi} \int_\R \frac{|F(u)/E(u)|}{(x-u)^2+y^2} du
\]
for $y>0$, and hence
\[
\int_\R |F(x+iy)/E(x+iy)|^pdx \le \int_\R |F(x)/E(x)|^p dx
\]
 which implies that $F/E \in H^p(\C^+)$. 
\end{proof}

As is common, we denote by $\varphi:\R \to \C$ a continuous, increasing function with the property that
\[
e^{i\varphi(x)} E(x)\in \R
\]
for all real $x$. The function $\varphi$ is called a {\it phase} of $E$; it is unique up to a constant multiple of $\pi$ (cf.\ \cite[Problem 48]{B}). We will make use of the fact that $e^{2i\varphi(x)}$ is the boundary function of $E^*/E\in H^\infty(\C^+)$.

Assume that $E$ has no real zeros. Since the phase of $\varphi$ is monotonically increasing, it follows that $A$ and $B$ have only real zeros, their zeros are simple, and they interlace. In particular, these functions have no common zeros. Hence, ${\rm sgn}(A B)$ has sign changes located exactly at the simple zeros of $A B$. It will be important below that $AB/|E^2|$ is bounded on the real line.  

For easy reference we record the following classical fact from Hardy space theory.

\begin{lemma}\label{H1-mean} If $F\in H^1(\C^+)$, then 
\[
\int_\R F(x) dx =0.
\]
\end{lemma}

\begin{proof} This follows since $\widehat{F}(t) =0$ for $t<0$ and the Fourier transform of an $L^1$-function  is continuous.
\end{proof}


\section{Extremal signatures}\label{es-section}

We consider in this section extremal signatures for $\mc{B}_1(\mu,2\tau)$. The change from $\tau$ to $2\tau$ simplifies notation, since most of the calculations for $\mc{B}_1(\mu,2\tau)$ will be done in $\mc{B}_2(\mu,\tau)$. 

We denote by $\mc{C}^+(\tau)$ the class of entire functions $M$ with the following properties.
\begin{itemize}
\item[(i)] $M$ is an entire function of exponential type $\tau$,
\item[(ii)] $M$ is positive on the real line,
\item[(iii)] The function $M$ satisfies
\begin{align}\label{es-krein-condition}
\int_\R \frac{\log^+|M(x)|}{x^2+1} dx <\infty.
\end{align}
\end{itemize}

Let $\tau_0\ge 0$. Let $M\in\mc{C}^+(2\tau_0)$.     It follows from Lemma \ref{akhiezer} that $M = UU^*$ where $U\in\mc{C}$ has exponential type $\tau_0$ and zeros only in the open lower half plane. For $\tau\ge \tau_0$  we define $E_{\tau, \alpha}$ by
\begin{align}\label{E-def}
E_{\tau,\alpha}(z) = U(z) e^{-i(\tau-\tau_0) z - i\alpha}.
\end{align}
 
 Let $A_{\tau, \alpha},B_{\tau,\alpha}$ be the unique real entire functions of exponential type $\le \tau$ with $E_{\tau,\alpha} = A_{\tau,\alpha} -iB_{\tau,\alpha}$.  Recall that $\mu_M$ is given by
\[
\mu_M(A) = \int_A \frac{dx}{M(x)}.
\]

\begin{theorem}\label{es-theorem}  Let $M\in \mc{C}^+(2\tau_0)$, and let $\tau\ge \tau_0$.  Then $\psi ={\rm sgn}(A_{\tau,\alpha} B_{\tau,\alpha})$ is an extremal signature for $\mc{B}_1(\mu_M,2\tau)$. 
\end{theorem}

\begin{proof} Throughout this proof we set $\mu = \mu_M$ and $E = E_{\tau,\alpha}$ with $E=A-iB$. By construction $U$ (and hence $E^2$) is an HB function that has bounded type in $\C^+$. Let $F\in\mc{B}_1(\mu,2\tau)$.  Lemma \ref{HpE-Lp-equivalent} implies that $F\in \mc{H}^1(E^2)$, and Lemma \ref{HpE-HpC-equivalent} implies that
\begin{align}\label{F-Hardy-condition}
F/E^2\in H^1(\C^+).
\end{align}

Let $N\in\N$ and define $S_N:\C\to \C$ by
\[
S_N(z) = \frac{4}{\pi i}\sum_{k=0}^{N-1} \frac{1}{2k+1}\left( \frac{E^*(z)}{E(z)}\right)^{2k+1}.
\]

Since $E$ is Hermite-Biehler, it follows that $E^*/E\in H^\infty(\C^+)$. The identity
\[
\frac{FS_N}{EE^*} = \frac{4F}{\pi i E^2} \sum_{k=0}^{N-1} \frac{1}{2k+1}\left( \frac{E^*}{E}\right)^{2k}
\]
implies with \eqref{F-Hardy-condition} that $FS_N(E E^*)^{-1}\in H^1(\C^+)$ for all $N\in\N$. For real $x$ we have $M(x) = |E(x)|^2$. This and   Lemma \ref{H1-mean} imply that for all $F\in \mc{B}_1(\mu,2\tau)$ and all $N\in\N$
\[
\int_\R F(x) S_N(x) d\mu(x) = \int_\R \frac{F(x) S_N(x)}{E(x) E^*(x)} dx =0.
\]

Since $F^*\in \mc{B}_1(\mu,2\tau)$, the identity $F(z) = \frac12 (F(z) + F^*(z)) - i\cdot \frac{i}{2} (F(z) - F^*(z))$  shows that any element of $\mc{B}_1(\mu,2\tau)$ can be written as a complex linear combination of real entire functions in $ \mc{B}_1(\mu,2\tau)$. We obtain for all $F\in \mc{B}_1(\mu,2\tau)$ that
\begin{align}\label{FSNE-mean}
\int_\R F(x) \Re S_N(x) d\mu(x) =0.
\end{align}

Denote by $\varphi$ the phase function of $E$. Recall that $\varphi$ is continuous and monotonically increasing on $\R$, and 
\begin{align}\label{EE*-rep}
\frac{E^*(x)}{E(x)} = e^{2 i \varphi(x)}
\end{align}
for all real $x$.   This allows us to rewrite \eqref{FSNE-mean} as
\[
 \int_\R F(x) \sum_{k=0}^{N-1} \frac{4\sin((4k+2)\varphi(x))}{\pi(2k+1)} d\mu(x)=0.
\]

The inner sum converges to ${\rm sgn}\sin2\varphi(x)$ and is uniformly bounded in $x$ and $N$. Hence dominated convergence implies that
\[
 \int_\R F(x)\, {\rm sgn}\sin(2\varphi(x)) d\mu(x) =0,
\]
which means that ${\rm sgn}\sin 2\varphi \in \mc{S}(\mu,2\tau)$. Finally, \eqref{EE*-rep} implies that $\sin 2\varphi = 2|E|^{-2}A B$ on $\R$,
which finishes the proof.
\end{proof}

By way of an example we consider
\[
d\nu(x) = \frac{dx}{x^2+1}.
\]

We may take $U(z) = z+i $, which means that  $\tau_0 =0$. This gives $E_{\tau,\alpha}(z) = (z+i) e^{-i\tau z - i \alpha}$ and hence
\begin{align*}
A_{\tau,\alpha}(z) &=  z \cos(\tau z +\alpha) + \sin(\tau z +\alpha)\\
B_{\tau,\alpha}(z) &= z \sin(\tau z+\alpha) -\cos(\tau z+\alpha).
\end{align*}

We would like to emphasize that the factorization of $M$ leads to an explicit construction of an entire function $A_\tau B_\tau$ whose real zeros are the points of sign change for an extremal signature. Combined with the interpolation procedure of the next section this allows constructions of best approximations.


\section{LP functions and interpolation}\label{LP-section}

A Laguerre-P\'{o}lya function  is a real entire function $L$ with Hadamard factorization given by
\begin{equation}\label{hadamardproduct}
L(z)=\frac{L^{(r)}(0)}{r!} z^r e^{-az^2+bz}\prod_{j=1}^\infty\left(1-\frac{z}{\xi_j}\right)e^{z/\xi_j},
\end{equation}
where $r\in\Z^{+}$, $a, b, \xi_j \in \R$, with $a \geq 0$, $\xi_j \neq 0$ and 
\[
\sum_{j=1}^\infty \xi_j^{-2}<\infty
\]
(the product may be finite or empty). We say that $L$ has degree $N = N_L$, with $0 \leq N< \infty$, if $a=0$ in \eqref{hadamardproduct} and $L$ has exactly $N$ zeros counted with multiplicity. Otherwise we set $N_L = \infty$.

\smallskip

For our purposes we require the property of Laguerre-P\'{o}lya functions that in vertical strips their reciprocals can be written as Laplace transforms. In fact, if $(\tau_1, \tau_2) \subset  \R$ is an open interval not containing any zeros of $L$ (we allow $\tau_1,\tau_2\in \{\pm\infty\}$) and $c \in (\tau_1, \tau_2)$, we have 
\begin{equation}\label{gc-trafo}
\frac{1}{L(z)} = \int_{-\infty}^\infty g_c(t) e^{-zt} dt
\end{equation}
for $\tau_1 < \Re(z) < \tau_2$, where 
\begin{equation}\label{gc-def}
g_c(t) = \frac{1}{2\pi i }\int_{c-i\infty}^{c+i\infty} \frac{e^{st}}{L(s)} ds.
\end{equation}
The integral \eqref{gc-def} is absolutely convergent if $N \ge 2$, and is understood as a Cauchy principal value if $N =1$. If $N =0$, \eqref{gc-trafo} holds with $g_c$ being an appropriate Dirac delta distribution. A simple application of the residue theorem gives us that $g_c = g_d$ if $c,d \in (\tau_1, \tau_2)$. An account of this theory can be found in \cite[Chapters II-V]{HW}.

Throughout this section we denote by $\alpha_L$ the smallest positive zero of $F$. (We set $\alpha_L=\infty$ if no such zero exists.) Analogously $\beta_L$ is the largest nonpositive zero of $L$.


\subsection{Approximation to the truncation of the exponential function}\label{trunc-exp} The following construction was developed in \cite{CL} in relation with the Beurling-Selberg extremal problem. Let $L$ be a Laguerre-P\'olya function with $L(\alpha_F/2) >0$. (If $\alpha_L =\infty$, assume $L(1)>0$.) Let $g = g_{\alpha_L/2}$ ($g=g_1$ if $\alpha_L =\infty$), and define for $\lambda>0$
\begin{align}\label{A-representation}
\begin{split}
I_1(L,\lambda, z) &= L(z) \int_{-\infty}^0 e^{-zu} g(u-\lambda) du \text{ if }\Re z<\alpha_L,\\
I_2(L,\lambda,z) &= e^{-\lambda z} - L(z) \int_0^\infty e^{-zu}  g(u-\lambda) du \text{ if }\Re z>\beta_L.
\end{split}
\end{align}

Morera's theorem implies that  that $z\mapsto I_1(L,\lambda,z)$ is analytic in $\Re z<\alpha_L$ and $z\mapsto I_2(L,\lambda,z)$ is analytic in $\Re z>\beta_L$. Multiplication of \eqref{gc-trafo} by $L(z) e^{-\lambda z}$, a change of variable in the right hand side, and inserting the resulting identity in \eqref{A-representation} shows that $I_1(L,\lambda,z) = I_2(L,\lambda,z)$ for $\beta_L <\Re z<\alpha_L$, and analytic continuation implies that these functions are restrictions to their respective half planes of an entire function in $z$ that we will call $I(L,\lambda,z)$. Moreover, estimation of the integrals in \eqref{A-representation} implies the existence of $c>0$ so that 
\begin{align}\label{complex-growth}
|I(L,\lambda, z)|\le c(1+|L(z)|)
\end{align}
 for all $z\in\C$.

To simplify notation we set
\[
x_+^0 = \begin{cases} 1 \text{ if }x\ge 0, \\ 0 \text{ if } x<0.\end{cases}
\]

Assume that $L(0)=0$. We define an entire function $z \mapsto I^\circ(L,\lambda,z)$ by
\[
I^\circ(L,\lambda,z) = I(L,\lambda,z) - g(-\lambda)\frac{L(z)}{z}.
\]

\begin{lemma}\label{AB-properties} Let $\lambda>0$ and let $L$ be an LP function with $L(0)=0$ and $L(\alpha_F/2)>0$.  Then 
\begin{align}\label{sign-of-difference}
L(x) \left( I^\circ(L,\lambda,x) - x_+^0 e^{-\lambda x} \right) \le 0
\end{align}
for all real $x$, and $I^\circ(L,\lambda,\xi) = \xi_+^0 e^{-\lambda \xi}$ for all $\xi\neq 0$ with $L(\xi) =0$. Moreover, there exists $C>0$ so that for all real $x$ the inequality
\begin{align}\label{B-integrable}
|I^\circ(L,\lambda,x) - x_+^0 e^{-\lambda x}|\le C\frac{|L(x)|}{1+x^2}
\end{align}
is valid.
\end{lemma}

\begin{proof}  It follows from  \eqref{A-representation} that
\begin{align*}
I^\circ(L,\lambda, x) - x_+^0 e^{-\lambda x}  =
\begin{cases}
\displaystyle - L(x) \int_{0}^\infty e^{- xu} (g(u-\lambda) - g(-\lambda)) du&\text{ if }x>0,\\[1.5ex]
\displaystyle L(x) \int_{-\infty}^0 e^{-xu} (g(u-\lambda) - g(-\lambda)) du&\text{ if }x<0.
\end{cases}
\end{align*}

Since both integrands equal zero at $u=0$ it follows that the integrals are $\mc{O}(|x|^{-2})$ for large $|x|$. We also get immediately that  $I^\circ(L,\lambda,\xi) = \xi_+^0 e^{-\lambda \xi}$ for all $\xi\neq 0$ with $L(\xi) =0$. Inequality \eqref{sign-of-difference} follows once we show that $g$ is monotonically increasing on $\R$. 

We assume first that the zero of $L$ at the origin is simple. Then $L(z)/z$ extends to an LP function that is positive at the origin. An integration by parts in \eqref{gc-trafo} gives
\[
\frac{z}{L(z)} = \int_\R e^{-zt} g'(t) dt
\]
in an open vertical strip $\beta<\Re z<\alpha$ with $\beta<0<\alpha$. Hence $g'$ is a totally positive function in the sense of \cite{HW}. By a classical theorem of I.J.\ Schoenberg \cite{Sch} (see also \cite[page 91]{HW}) the $n$th derivative of a totally positive function (if it exists) has exactly $n$ sign changes on the real line. Applying this to $L(z)/z$ and $g'$ shows in particular that $g'$ has no sign changes on the real line, and it must be positive since $L(z)/z$ is positive at the origin. This means that $g$ is monotonically increasing. 

If the zero of $L$ at the origin has higher multiplicity, we iterate this argument.  
\end{proof}

\begin{theorem} Let $M\in\mc{C}^+(2\tau_0)$, let $\tau\ge \tau_0$, and let $E_{\tau,\alpha}$ be defined by \eqref{E-def}.  Assume in addition that $B_{\tau,\alpha}(0) =0$ and $A_\tau B_\tau$ is positive in some interval with left endpoint at the origin. Then for all $F\in \mc{B}_1(\mu_M,2\tau)$ 
\[
\int_\R \left|F(x) - x_+^0e^{-\lambda x} \right| d\mu_M(x)  \ge \int_0^\infty e^{-\lambda x} {\rm sgn}(A_\tau(x) B_\tau(x)) d\mu_M(x)
\]
with equality if $F(z) = I^\circ(A_\tau B_\tau, \lambda, z)$. 
\end{theorem}

\noindent{\it Remark.} For given $M$ the conditions on $A_{\tau,\alpha}$ and  $B_{\tau,\alpha}$ may be satisfied by a suitable choice of $\alpha$. 

\begin{proof} We set $\mu=\mu_M$ and $E = E_{\tau,\alpha}$. It follows from Theorem \ref{es-theorem} that ${\rm sgn}(A B)$ is an extremal signature for $\mc{B}_1(\mu,2\tau)$. By construction $I^\circ(A B, \lambda,x) -x_+^0 e^{-\lambda x}$ has its sign changes on the real line exactly at the zeros of $A B$.

Since $E$ has exponential type $\tau$, it follows that $A B$ has exponential type $2\tau$. We obtain from \eqref{complex-growth} that $I^\circ(A B, \lambda, z)$ has the same exponential type. Since $A B/M$ is bounded on $\R$, it follows from \eqref{B-integrable} that 
\[
\int_\R \left| I^\circ(A B, \lambda, x) - x_+^0e^{-\lambda x}\right|d\mu<\infty.
\]

The identity
\[
| I^\circ(A B, \lambda, x) - x_+^0e^{-\lambda x}| = {\rm sgn}(A(x) B(x))( I^\circ(A B, \lambda, x) - x_+^0e^{-\lambda x})
\]
and \eqref{intro-es} for $\psi = {\rm sgn}(AB)$ gives
\[
\int_\R \left|I^\circ(A B, \lambda, x) - x_+^0e^{-\lambda x}\right|d\mu(x) =  \int_0^\infty e^{-\lambda x} {\rm sgn}(A(x) B(x)) d\mu(x).
\]

The inequality for $F\in \mc{B}_1(\mu,2\tau)$ follows from Theorem A.
\end{proof}

\subsection{Structure of $E_{\tau,\alpha}$ for even measures} Best approximations to the Gaussian $e^{-\lambda x^2}$ and to $e^{-\lambda |x|}$ rely on interpolation formulas that require the function $A_\tau B_\tau$ to be even. We show next that if $M$ is even, then $\alpha_0$ exists so that  $A_{\tau,\alpha_0} B_{\tau,\alpha_0}$ is even.

Assume that $M$ is even. Since an entire function of finite exponential type is a function of order $1$ and normal type in the language of Levin, Hadamard's factorization theorem (e.g., \cite[page 24]{Le}) gives
\[
M(z) = e^{az + b} \prod_{\xi \in\mc{T}_M} \left(1-\frac{z}{\xi}\right) e^{z/\xi}
 \] 
 where $\mc{T}_M$ is the set of zeros of $M$ listed with multiplicity. We note that $\mc{T}_M$ contains no real numbers. Since $M$ is even and positive on the real line, $\xi\in\mc{T}_M$ implies that $-\xi, \overline{\xi}, -\overline{\xi}\in \mc{T}_M$ with the same multiplicity. We note that if $\Im \xi <0$, then $-\Im \overline{\xi}<0$, and the condition that $M$ is positive on the real line implies that the first exponential factor is of the form $e^{\alpha z+\beta}$ with real $\alpha,\beta$, and since $M$ is even, $\alpha=0$.
 
It follows that the function $U$ defined by
 \[
 U(z) =  e^{\beta/2} \prod_{\substack{\xi \in \mc{T}_M \\ \Im \xi <0}} \left(1-\frac{z}{\xi}\right) e^{z/\xi}
 \]
 satisfies $UU^* = M$ and $ U^*(z) = U(-z)$. We obtain $E_{\tau,0} (-z) = E_{\tau,0}^*(z)$.  A direct calculation gives
 \begin{align}\label{AB-even}
 A_{\tau,\frac\pi4}(z) B_{\tau,\frac\pi4}(z) = -\frac14 \left(E_{\tau,0}(z)^2 + E_{\tau,0}^*(z)^2\right),
 \end{align}
which implies that $ A_{\tau,\frac\pi4} B_{\tau,\frac\pi4}$ is even.


\subsection{Approximation to $e^{-\lambda |x|}$} 
 For a Laguerre-P\'olya entire function $L$ we define $I^\#$ by
\begin{align}
\begin{split}
I^\#(L, \lambda,z) &= I(L, \lambda,z) + I(L, \lambda,-z),
\end{split}
\end{align}
where $I(L,\lambda,z)$ is defined in Section \ref{trunc-exp}.

The necessary interpolation result for this case may be obtained from \cite[Proposition 8]{CL} where it is shown that for $\lambda>0$ and $L$ an even LP function with $ F(0)>0$ the inequality
\[
L(x) \left( e^{-\lambda |x|} - I^\#(L,\lambda,x)\right) \ge 0
\]
is valid for all real $x$. Moreover, it is shown that $I^\#(L,\lambda, \xi) =e^{-\lambda|\xi|}$ for all $\xi$ with $L(\xi) =0$. An application of this inequality combined with the above construction of even $L=A_{\tau,\pi/4} B_{\tau,\pi/4}$ (which necessarily must have a nonzero value at the origin, since its zeros are all simple, see Section \ref{dBspace}) gives the following result. 

\begin{theorem} Let $M\in\mc{C}^+(2\tau_0)$ be even, and let $\tau\ge \tau_0$.  Then for all $F\in \mc{B}_1(\mu_M, 2\tau)$
\[
\int_\R |F(x) - e^{-\lambda |x|} |d\mu_M(x) \ge \int_\R e^{-\lambda |x|} {\rm sgn}(A_{\tau,\frac\pi4}(x) B_{\tau,\frac\pi4}(x)) d\mu_M(x)
\]
with equality if $F(z) = I^\#(A_{\tau,\frac\pi4} B_{\tau,\frac\pi4}, \lambda, z)$. 
\end{theorem}


\subsection{Approximation to the Gaussian} Let $M(x)$ be even. As is shown in \cite{CL2}, an interpolation to the Gaussian can be obtained in the following way. For an even LP function $L$, the Hadamard factorization may be used to construct an LP function $G_L$ (of exponential type zero) with
\[
G_L(z^2) = L(z).
\]

Similar to the proof of Lemma \ref{AB-properties} it can be shown that $L(x) (I(L,\lambda,x) - x_+^0 e^{-\lambda x})\ge 0$ for all positive $x$. Applying this with $G_L$ instead of $L$ and substituting $x^2$ for $x$, it follows that
\[
G_L(x) (I(G_L,\lambda, x^2) - e^{-\lambda x^2}) \ge 0
\]
for all real $x$. Properties of $I(G_L,\lambda,z^2)$ analogous to the properties of $I^\circ(L,\lambda,z)$   are proved in \cite{CL2}. They imply the following statement.

\begin{theorem} Let  $M\in\mc{C}^+(2\tau_0)$ be even, and let $\tau\ge \tau_0$.  Then for all $F\in \mc{B}_1(\mu_M, 2\tau)$
\[
\int_\R |F(x) - e^{-\lambda x^2} |d\mu_M(x) \ge \int_\R e^{-\lambda x^2} {\rm sgn}(A_{\tau,\frac\pi4}(x) B_{\tau,\frac\pi4}(x)) d\mu_M(x)
\]
with equality if $F(z) = I( G_{A_{\tau,\frac\pi4} B_{\tau,\frac\pi4}}, \lambda, z^2)$. 
\end{theorem}


\subsection{Interpolation of the Poisson kernel}

For $\lambda>0$ we define the Poisson kernel $P_\lambda$ and conjugate Poisson kernel $Q_\lambda$ by
\begin{align*}
P_\lambda(z) &= \frac{\lambda}{\pi(z^2+\lambda^2)},\\
Q_\lambda(z)&= \lambda^{-1} z P_\lambda(z).
\end{align*}

Due to the fact that $P_\lambda$ and $Q_\lambda$ are already meromorphic, the best approximations and the error representations have particularly simple expressions. Let $E$ be an entire function with $E(i\lambda)\neq 0$. (This is certainly satisfied if $E$ is HB.) We define $K_{\lambda, E}$ and $L_{\lambda, E}$ by
\begin{align}\label{K-def}
\begin{split}
K_{\lambda,E}(z) &= P_\lambda(z) \left(1- \frac{E(z)^2+E^*(z)^2}{E(i\lambda)^2+E^*(i\lambda)^2}\right),\\
L_{\lambda,E}(z)&= \lambda^{-1} P_\lambda(z)\left(z-i\lambda\frac{E(z)^2- E^*(z)^2}{E(i\lambda)^2 - E^*(i\lambda)^2}\right)
\end{split}
\end{align}
where $z\in\C$. If $E(-z) = E^*(z)$ for all $z\in\C$, then evidently $K_{\lambda, E}$ and $L_{\lambda,E}$ are entire functions. If $E$ has exponential type $\tau$, then $K_{\lambda, E}$ and $L_{\lambda,E}$ have exponential type $2\tau$. We emphasize that in the first part of the following theorem the interpolation is obtained by taking $E = E_{\tau,0}$, while the extremal signature being used is ${\rm sgn}(A_{\tau,\frac\pi4} B_{\tau,\frac\pi4})$. 

\begin{theorem}\label{poisson-approx} Let $M\in\mc{C}^+(2\tau_0)$, and let $\tau\ge \tau_0$. Assume that $E_{\tau,0}(-z) = E_{\tau,0}^*(z)$. Then the following statements hold.
\begin{enumerate}
\item For all $F\in  \mc{B}_1(\mu_M, 2\tau)$
\[
\int_\R |P_\lambda(x) - F(x)| d\mu_M(x) \ge \frac{4}{\pi E_{\tau,0}(i\lambda) E_{\tau,0}(-i\lambda)} \arctan\left(\frac{E_{\tau,0}(-i\lambda)}{E_{\tau,0}(i\lambda)}\right) 
\]
with equality if $F = K_{\lambda, E_{\tau,0}}$. 
\item  For all $F\in  \mc{B}_1(\mu_M, 2\tau)$
\[
\int_\R |Q_\lambda(x) - F(x)|d\mu_M(x) \ge \frac{4}{\pi E_{\tau,0}(i\lambda) E_{\tau,0}(-i\lambda)} {\rm arctanh}\left(\frac{E_{\tau,0}(-i\lambda)}{E_{\tau,0}(i\lambda)}\right),
\]
with equality if $F = L_{\lambda, E_{\tau,0}}$. 
\end{enumerate}
\end{theorem}

\begin{proof} The identity $E_{\tau,0}(-z) = E_{\tau,0}^*(z)$ implies that $E(i\R)\subseteq \R$. Hence $E_{\tau,0}(i\lambda)^2+E_{\tau,0}^*(i\lambda)^2\ge 0$. It follows from \eqref{AB-even} and \eqref{K-def} that
\[
A_{\tau,\frac\pi4}(x)B_{\tau,\frac\pi4}(x)( K_{\lambda, E_{\tau,0}}(x) - P_\lambda(x)) \ge 0
\]
for all real $x$. Integrability of $K_{\lambda,E_{\tau,0}} - P_\lambda$ with respect to $\mu_M$ follows from \eqref{K-def} as well as the fact that $K_{\lambda, E_{\tau,0}}$ has exponential type $2\tau$. It remains to show that
\[
\int_\R |K_{\lambda, E_{\tau,0}}(x) - P_\lambda(x)| d\mu_M(x) =  \frac{4}{\pi E_{\tau,0}(i\lambda) E_{\tau,0}(-i\lambda)} \arctan\left(\frac{E_{\tau,0}(-i\lambda)}{E_{\tau,0}(i\lambda)}\right).
\]

Since $E_{\tau,0}^*/E_{\tau,0}$ is bounded by $1$ in the upper half plane and its reciprocal is bounded in the lower half plane, an application of the residue theorem gives
\begin{align}\label{residue}
\int_\R \left(\frac{E^*_\tau(x)}{E_{\tau,0}(x)}\right)^k \frac{dx}{x^2+\lambda^2} = \begin{cases}
\displaystyle \pi\lambda^{-1} E_{\tau,0}(i\lambda)^{-k}E_{\tau,0}^*(i\lambda)^k&\text{ if }k\ge 0,\\
\displaystyle \pi\lambda^{-1} E_{\tau,0}(-i\lambda)^{-k} E_{\tau,0}^*(-i\lambda)^k & \text{ if }k<0.
\end{cases}
\end{align}

We note that
\[
{\rm sgn} (A_{\tau,\frac\pi4}(x) B_{\tau,\frac\pi4}(x)) = {\rm sgn}\cos 2\varphi_{\tau,0}(x)
\]
where $\varphi_{\tau,0}$ is the phase of $E_{\tau,0}$. The Fourier series of the sign of $\cos y$ gives
\[
 {\rm sgn}\cos 2\varphi_{\tau,0}(x) = \frac2\pi \sum_{n\in\Z} \frac{(-1)^n}{2n+1} \left(\frac{E_\tau^*(x)}{E_\tau(x)}\right)^{2n+1}.
\]

Multiplication by $K_{\lambda, E_\tau}(x) - P_\lambda(x)$, integration against $dx/(E_{\tau,0}(x) E_{\tau,0}^*(x))$ and multiple applications of \eqref{residue} give the evaluation. The proof for the conjugate Poisson kernel is analogous using the fact that $E_{\tau,0}^2 - (E_{\tau,0}^*)^2 = -4iA_{\tau,0} B_{\tau,0}$. 
\end{proof}



\section{Open problems}

There are two questions that are immediately suggested by the results of the previous sections. 
\begin{enumerate}
\item Find an effective characterization of all extremal signatures for $dx/M(x)$ where $M(x)$ is as in Theorem \ref{es-theorem}.
\item Extend the results of Theorem \ref{es-theorem} to more general measures.
\end{enumerate}

 In the case of Lebesgue measure, an explicit parametrization of all extremal signatures was found by B.\ Logan in his thesis \cite{Lo}. Starting from the fact that $\psi \in \mc{A}(dx, \tau)$ is equivalent to 
\[
|P_y*\psi(x)|\le A e^{-\tau y}
\]
for some $A>0$, all real $x$, and $y>0$, Logan showed that $h\circ \psi\in \mc{A}(dx,\tau)$ where $h$ is a periodic high pass function, and $\psi$ is essentially the logarithm of an inner function with exponential decay in the upper half plan and applied this with $h(x) = {\rm sgn} \sin \tau x$ to obtain his result. A full account of his argument can be found in  \cite[Chapter 7.6]{Sh}. For non-constant $M$ there are not even conjectures with regards to the correct formulation.

This type of representation would be useful to find best approximations to functions with more than one discontinuity, e.g., characteristic functions of intervals on the real line. 

Regarding the second question, it is frequently possible to find descriptions for high pass functions for a given measure, but these descriptions do not lend themselves to investigations of $\pm1$ functions. To give a simple example, consider the measure $x^2dx$. Define $\psi$ to be of the form
\[
\psi(x) = aj_0(x) +b j_1(x) + \psi_0(x)
\] 
where $j_0$ and $j_1$ are the spherical Bessel functions $j_0(x) = \sin x/x$ and $j_1(x) = (\sin x - x\cos x)/x^2$, $a,b\in\C$, and $\psi_0$ is any element in $\mc{A}(dx,1)$. It is easy to prove that $\int_\R F(x) j_0(x) x^2 dx =0$ and $\int_\R F(x) j_1(x) x^2 dx=0$ for every $F\in \mc{B}_1(x^2dx, 1)$, and hence $\psi\in \mc{A}(x^2 dx,1)$, but the additive structure of this representation is not well suited to investigate functions of absolute value $1$.

\end{document}